\documentclass[11pt,reqno,final]{article}

\usepackage{amsmath,amsfonts,amssymb,amsthm,version}
\usepackage{mathrsfs,fancybox,pifont}
\usepackage{graphicx}
\usepackage{algorithm,algorithmic}
\usepackage{cite}
\usepackage{setspace}
\usepackage{url,hyperref}
\usepackage{fullpage} 
\usepackage[notcite,notref]{showkeys}
\usepackage{xcolor}
\usepackage{subfigure,multirow}
\usepackage{epstopdf}
\usepackage{cases}
\usepackage{geometry}
\usepackage{mathtools}
\usepackage{authblk}
\usepackage{lipsum}
\usepackage{indentfirst}

\newcommand{\dd}{\operatorname{d}\! }
\newcommand{\dt}{\operatorname{d}\! t}
\newcommand{\ds}{\operatorname{d}\! s}
\newcommand{\dr}{\operatorname{d}\! r}

\newcommand{\dw}{\operatorname{d}\! W}

\newcommand{\R}{\mathbb{R}}
\newcommand{\nn}{\nonumber}

\DeclareMathOperator*{\essinf}{ess\,inf}

\renewcommand{\geq}{\geqslant}
\renewcommand{\leq}{\leqslant}
\renewcommand{\ge}{\geqslant}

\allowdisplaybreaks

\numberwithin{equation}{section}

\theoremstyle{plain}
\newtheorem{theorem}{Theorem}
\newtheorem{proposition}{Proposition}
\newtheorem{lemma}{Lemma}
\newtheorem{assumption}{Assumption}
\newtheorem{remark}{Remark}

\theoremstyle{definition}

\newtheorem{prob}{Problem}

\title{Stochastic LQ Optimal Control with Random Coefficients and a Terminal Mean-Field Cost\thanks{This paper is supported by National Key R\&D Program of China (No.2022YFA1006101), National Natural Science Foundation of China (Nos.12371445, 12571517), Hong Kong RGC (GRF 15203423, 15204622), PolyU-SDU Joint Research Center on Financial Mathematics, CAS AMSS-PolyU Joint Laboratory of Applied Mathematics, Research Centre for Quantitative Finance (1-CE03), internal grants from The Hong Kong Polytechnic University, and State Key Laboratory of Cryptography and Digital Economy Security, Shandong University (No.KFZD2505).}}

\author{Guojiang Shao\thanks{School of Mathematical Sciences, Fudan University, Shanghai 200433, China. Email: \url{gjshao23@m.fudan.edu.cn}},~~ Zuo Quan Xu\thanks{Department of Applied Mathematics, The Hong Kong Polytechnic University, Kowloon, Hong Kong SAR, China. Email: \url{maxu@polyu.edu.hk}}, ~~and Qi Zhang\thanks{Corresponding author. School of Mathematical Sciences, Fudan University, Shanghai 200433, China; State Key Laboratory of Cryptography and Digital Economy Security, Shandong University, Jinan 250100, China. Email: \url{qzh@fudan.edu.cn}}}

\begin{document}
\maketitle
\begin{abstract}
This paper investigates a multidimensional non-homogeneous stochastic linear-quadratic optimal control problem featuring random coefficients and a terminal mean-field term in the cost functional, enabling its direct application to mean-variance models in financial engineering. Employing the Lagrangian duality method together with a decomposition approach for linear backward stochastic differential equations, we provide two types of sufficient conditions for solvability and derive the corresponding optimal controls. In particular, in the deterministic-coefficient case, our condition is weaker than the standard condition found in the existing literature on mean-field stochastic LQ problems. Finally, a numerical example drawn from optimal portfolio selection with multiple assets under mean-variance utility demonstrates the applicability of our results.
\end{abstract}

\section{Introduction}\label{sec:introduction}

Let $(\Omega, \mathcal{F}, \mathbb{P})$ be a complete probability space, and $\{W_t\}_{t\geq0}$ be a one-dimensional standard Brownian motion in it with its natural filtration augmented by all $\mathbb{P}$-null sets, and set $\mathbb{F}=\{\mathcal{F}_t\}_{t \geq 0}$. The state dynamics we concern in this paper is the following $ \R^n$-valued controlled linear stochastic differential equation (SDE)
\begin{equation} \label{staeq}
\dd X(s) = \{A(s) X(s) + B(s) u(s) + a(s) \} \ds + \{C(s) X(s) + D(s) u(s) + b(s)\} \dw(s), ~~ s \in[0, T], 
\end{equation}
and the quadratic cost functional given by
\begin{align} \label{functional}
J(x ; u) :=&~ \mathbb{E}^{0,x,u} \bigg\{\int_0^T \left[ \langle Q(s) (X(s)-q(s)), X(s)-q(s) \rangle +\langle R(s)( u(s)-p(s)), u(s)-p(s) \rangle \right] \ds \nn\\
& + \langle G (X(T)-\xi),X(T)-\xi\rangle + \langle\bar{G} (\mathbb{E}^{0,x,u}[X(T)]-\eta), \mathbb{E}^{0,x,u}[X(T)]-\eta \rangle\bigg\},
\end{align}
%
%
where the control process $u$ and the coefficients 
satisfy suitable conditions which will be specified in next section. Here $\mathbb{E}^{t,x,u}$ denotes the expectation with respect to the law of the state process $\{X_s\}_{s\in[t,T]}$ governed by \eqref{staeq} with initial state $X(t)=x$ and control $u$.

Define the admissible control set
\begin{equation*}
\mathcal{U}[t, T]:= \left\{u:[t, T] \times \Omega \rightarrow \R^m, ~ \mathbb{F} \text {-progressively measurable and}\ \mathbb{E} \left[ \int_t^T\|u(s)\|^2 \ds \right] <\infty\right\}.
\end{equation*}
The stochastic linear-quadratic (SLQ) optimal control problem with random coefficients and a terminal mean-field cost is formulated as follows.
\begin{prob} \label{MFLQ}
Given $x \in \R^n$, find a $\bar{u} \in \mathcal{U}[0, T]$ such that
\begin{equation*}
J\left(x ; \bar{u}\right)=\inf _{u \in \mathcal{U}[0, T]} J(x ; u).
\end{equation*}
\end{prob}
If it exists, \(\bar{u}\) is called an \textbf{optimal control}; the corresponding state process \(\bar{X}\) governed by \eqref{staeq} called the \textbf{optimal state process}; and \((\bar{X}, \bar{u})\) called the \textbf{optimal pair}. \textbf{Problem \ref{MFLQ}} is said to be \textbf{solvable} if it admits an optimal pair.

{\bf Problem \ref{MFLQ}} has a wide range of applications. Here is an example. Consider the problem of optimal portfolio selection under the extended mean-variance (MV) criterion in a non-Markovian financial market consisting of $n$ risky assets (stocks). The price process of asset $i$, $S_i=\{S_i(t)\}_{t\in[0,T]}$, is governed by
\begin{equation}\label{assets}
	\frac{ \dd S_i(s)}{S_i(s)} = \mu_i(s)\ds + \sigma_i(s)\dw(s), \quad S_i(0) = s_i>0, \quad s \in[0, T],
\end{equation}
where $\mu_i$ is the appreciation rate process of asset $i$ and $\sigma_i$ represents the volatility process with suitable conditions, $i=1,2,\cdots,n$.
Suppose an investor manages the $n$ assets in $n$ different accounts, and the corresponding wealth process is a vector process $X = \{(X_1(t),X_2(t) \cdots, X_n(t))^\top \}_{t\in[0,T]}\in \R^n$. If transaction costs and consumption are ignored and share trading takes place in continuous time, then the wealth process for asset $i$, $\{X_i(t)\}_{t\in[0,T]}$, satisfies
\begin{equation*}
	\left\{\begin{aligned}
		 \dd X_i(s) &= \mu_i(s)\pi_i(s)\ds + \sigma_i(s)\pi_i(s) \dw(s), \quad s \in[0, T],
		\\ X_i(0)&=x_i\in \R,
	\end{aligned}\right.
\end{equation*}
where $\pi = \{(\pi_1(s), \cdots, \pi_n(s))^\top \}_{s\in[0,T]}\in \R^n$ is the amount invested in the corresponding risky asset of each account, $i=1,2,\cdots,n$.
The investor's preference can be formulated through the utility functional
\begin{equation}\label{preference}
	J(x;\pi) := \mathbb{E}^{0,x,\pi} \langle \upsilon, X(T) \rangle - \mathbb{E}^{0,x,\pi} \langle \Sigma \left( X(T)-\mathbb{E}^{0,x,\pi}[X(T)] \right), X(T)-\mathbb{E}^{0,x,\pi}[X(T)] \rangle.
\end{equation}
Here the vector $v \in \mathbb{R}^n$ represents the investor's preference toward expected returns across different accounts. Each component $v_i$, $i=1,2,\ldots,n$, measures the relative importance assigned to the expected terminal wealth of asset $i$. The positive definite matrix $\Sigma \in \mathbb{R}^{n \times n}$ characterizes the investor's risk attitude in a multi-dimensional manner. The diagonal elements reflect the degree of risk aversion toward individual assets, while the off-diagonal elements capture the investor's sensitivity to joint fluctuations between different assets, thereby encoding preferences toward diversification or concentration. Therefore, compared with classical MV models featuring a scalar risk aversion parameter, the above extended MV model, indexed by $(v, \Sigma)$, provides a flexible framework to describe heterogeneous preferences and cross-asset risk interactions.

Since the pioneering work of Wonham \cite{wonham1968matrix} in 1968, the LQ optimal control problem has been one of the most important topics in control theory, owing to its special structure and wide range of applications. Bismut \cite{bismut1976linear} in 1976 was the first to investigate the SLQ problem with random coefficients. Kohlmann and Zhou \cite{kohlmann2000relationship} established a fundamental connection between SLQ problems and backward stochastic differential equations (BSDEs) in 2000. It is well known that the stochastic Riccati equation (SRE) plays a crucial role in deriving a closed-form representation of the optimal control in SLQ problems. In the classical setting with stochastic coefficients, it is actually a BSDE of the form
\begin{equation} \label{bsde1}
\left\{\begin{aligned}
& \dd P = - g_1(P,\Lambda) \ds + \Lambda \dw, \quad s \in[0, T], \\
& P(T)= G,
\end{aligned}\right.
\end{equation}
where
\begin{equation*}
g_1(P,\Lambda) = -\left(\Theta(P,\Lambda) \right)^{\top}\left(D^{\top} P D + R\right)\Theta(P,\Lambda) + PA + A^{\top}P + \Lambda C + C^{\top}\Lambda + C^{\top} P C + Q,
\end{equation*}
and $\Theta(P,\Lambda)$ is defined as
\begin{equation}\label{thetafun}
\Theta(P,\Lambda) := -\left(D^{\top} P D + R\right)^{-1} \left( B^{\top} P +D^{\top}\Lambda + D^{\top} P C \right).
\end{equation}
As is well-known, its solution consists of a pair $(P,\Lambda)$, which is called uniformly positive if so is $P$.
Note the SRE \eqref{bsde1} will reduce to an ordinary differential equation (ODE) when the coefficients are all deterministic, even if the LQ problem is formulated in a stochastic environment.

The solvability of SRE had been regarded as ``a challenging task" for a long time due to its highly nonlinear generator. Kohlmann and Tang \cite{kohlmann2002global} first established the existence and uniqueness of one-dimensional SREs in 2002. Soon after, Tang \cite{tang2003general} made a breakthrough in 2003 by proving the existence and uniqueness of matrix-valued SREs under the assumption of a uniformly positive control weighting matrix (referred to as the standard case in the literature). Subsequently, Tang \cite{tang2015dynamic} provided an alternative proof via the dynamic programming principle in 2015.
To some extent, the uniformly positive condition can be relaxed. Considerable work has been done on SLQ problems with an indefinite control weighting matrix (referred to as the singular case), including contributions by Qian and Zhou \cite{qian2013existence}, Du \cite{du2015solvability}, and Sun, Xiong and Yong \cite{sun2021indefinite}.

As for the uniformly positive definiteness of the solution component $P$ to the SRE in \eqref{bsde1}, for which the positive definiteness of $G$ is necessary, the well-posedness and uniform positive definiteness of the solution to the multi-dimensional SRE were proved by Kohlmann and Tang \cite{kohlmann2003minimization} in 2003 under the condition that $D^\top D$ is uniformly positive definite, based on a priori estimates and the results in \cite{tang2003general}. However, under the condition that $R$ is uniformly positive definite, only the uniform positive definiteness of the solution to the one-dimensional SRE was recently obtained by Hu, Shi, and Xu \cite{MR4729330} with the help of the comparison theorem for quadratic BSDEs. To the best of our knowledge, no existing results on the uniform positive definiteness of the solution are available for multi-dimensional SREs when only the uniform positive definiteness of $R$ is assumed, due to the absence of a comparison theorem in the multi-dimensional case. One motivation of this paper is to prove this.

On the other hand, mean-field systems have become highly popular research topics in recent years, and mean-field control systems are no exception. Mean-field control problems concern the centralized optimization of large-population systems whose dynamics are governed by McKean-Vlasov-type SDEs, focusing on designing a social planner that minimizes a global cost functional depending on both individual trajectories and their distributions. Among mean-field control problems, the mean-field LQ optimal control problem is one of the most extensively studied due to its analytical tractability and wide applicability. However, the mean-field term causes significant difficulty for the classical Riccati equation method. Yong \cite{yong2013linear} in 2013 managed to solve it by introducing two Riccati equations --- one related to the expectation of the state and the other to the deviation of the state --- to provide an explicit optimal solution. Subsequently, more work on the mean-field LQ problem appeared. For example, Yong \cite{yo2} investigated the time-consistent properties of the optimal control in 2017; Sun \cite{su} studied the indefinite case and open-loop solvability of such systems in the same year; and Tian and Yu \cite{ti-yu} generalized the state equation to a forward-backward stochastic differential equation in 2023, to name but a few.

It should be noted, however, that in most papers on mean-field control problems, the coefficients of the control system are required to be deterministic, since the method in \cite{yong2013linear} does not work for random coefficients. Consequently, existing results on mean-field control problems with random coefficients are very scarce. While Pham \cite{pham2016linear}, Mei, Wei and Yong \cite{mei2023linear}, and Ding et al. \cite{DLLX25} considered mean-field control problems with random coefficients, the settings in their control systems are quite restrictive. Xiong and Xu \cite{xiong2024mean} advanced this direction by studying a relatively general form of mean-field control system with random coefficients, but still without a terminal mean-field term --- i.e., the expectation of the state at the terminal time. As a result, their results cannot be applied to the MV portfolio selection problem directly. This limitation motivates us to study the mean-field control problem with a terminal mean-field term.

As mentioned above, SLQ with a terminal mean-field term is usually associated with the MV portfolio selection problem, one of the core problems in mathematical finance. The pioneering work of Markowitz \cite{markowitz1952modern} in 1952 introduced MV analysis, which optimizes asset allocation by balancing risk and return. Since then, the continuous-time MV portfolio optimization problem has been extensively studied under various market assumptions and methodological frameworks. We briefly recall some important developments on this topic.

For deterministic coefficients, Li and Zhou \cite{zhou2000continuous} investigated the continuous-time MV portfolio selection problem in 2000 using the embedding technique and the well-developed SLQ theory with deterministic coefficients. Under the constraint that short-selling of stocks is prohibited, Li, Zhou and Lim \cite{li2002dynamic} solved the continuous-time MV portfolio selection problem in 2002 by means of the Hamilton-Jacobi-Bellman (HJB) equation and two Riccati equations. As for the random coefficients case, Lim and Zhou \cite{lim2002mean} studied the continuous-time MV problem in a complete market in 2002 based on SLQ theory and BSDE theory. Hu, Shi and Xu \cite{hu2022constrained} dealt with a constrained SLQ problem and an MV portfolio selection problem with regime switching and random coefficients in 2022. However, all of the above MV references are based on one-dimensional SLQ theory, regardless of whether the coefficients are deterministic or random. Needless to say, solving the multi-dimensional MV portfolio selection problem with random coefficients --- which corresponds to a multi-dimensional SLQ problem with random coefficients and a terminal mean-field cost --- remains a challenge.

Recently, Zhang and Zhang \cite{zhang2023stochastic} in 2023 considered a multi-dimensional SLQ problem with an expectation-type linear equality constraint on the terminal state, derived equivalent characterizations of surjectivity conditions, and applied their theoretical results to portfolio management.
They extended the result to the case with a point-wise linear equality constraint on the terminal state in \cite{zhang2024lagrangian}. However, in both \cite{zhang2023stochastic} and \cite{zhang2024lagrangian}, the expectation of the state at the terminal time appears as a constraint rather than as part of the cost functional. This constitutes an essential difference between the problems they considered and \textbf{Problem \ref{MFLQ}} regarding the solvability of the SLQ problem.

The main contributions of this paper are as follows.
\textbf{First}, we establish the uniform positive definiteness of the solution $P$ to the SRE under the condition that $R$ is uniformly positive definite, using a priori estimates different from those in Kohlmann and Tang \cite{kohlmann2002global}, \cite{kohlmann2003minimization}.
\textbf{Second}, we apply the Lagrange multiplier method to study a new type of multi-dimensional SLQ control problem with random coefficients and a terminal mean-field cost, by fixing the expectation of the state at the terminal time. To characterize the value function, we utilize two types of BSDEs and derive two types of sufficient conditions on $\bar{G}$ to guarantee the solvability of \textbf{Problem \ref{MFLQ}}, where the second type condition comprises three specific conditions. From the third specific sufficient condition, we observe that in some special cases --- such as the deterministic coefficients case --- the sum $\bar{G} + G$ may be slightly negative. Consequently, the condition on $\bar{G}$ is partially relaxed compared with the existing literature, in which $\bar{G} + G$ is required to be positive semidefinite.
\textbf{Finally}, we construct a financial model to demonstrate a practical application of our theory and conduct a detailed analysis of the impact of model parameters on utility.

The remainder of this paper is organized as follows. Section \ref{pre} introduces some notations and assumptions and formulates several bridge problems. Section \ref{sec3} solves the bridge problems and \textbf{Problem \ref{MFLQ}}. Specifically, Subsection \ref{positiveSRE} establishes the uniform positive definiteness of the solution to the SRE \eqref{bsde1}, and Subsection \ref{SREsolution} solves \textbf{Problem \ref{MFLQ}} and presents two types of sufficient conditions on $\bar{G}$ for solvability. Section \ref{numerical} provides numerical results for the deterministic coefficients case. Section \ref{conclude} offers some concluding remarks.

\section{Preliminaries and Bridge problems}\label{pre}

\subsection{Notation and spaces}

In this section, we introduce some notations and assumptions. For $n\in \mathbb{N}$, denote by $\langle \cdot, \cdot \rangle$ the inner product in $ \R^n$, by $\mathbb{S}^n$ the set of all $n \times n$ symmetric matrices, and by $\mathbb{S}_+^n$ the set of all $n \times n$ positive semidefinite matrices. 
We say a matrix process $M\in\mathbb{S}^n$ is uniformly positive, written as $M > 0$, if $M\geq \delta I_n$ for some constant $\delta>0$, where $I_n$ denotes the identity matrix in $ \mathbb{S}^n$, the definitions of $M < 0$ and $M_1 > M_2$ are similar. Let $\|\cdot\|$ denote the standard Frobenius norm for matrices.

The following function spaces will be used in the paper.

\begin{itemize}
\item $L_{\mathcal{F}_T}^{\infty}(\mathbb{H})$: the space of $\mathcal{F}_T$-measurable $\mathbb{H}$-valued essentially bounded random variables;
\item $L_{\mathcal{F}_{T}}^2(\mathbb{H})$: the space of $\mathcal{F}_{T}$-measurable $\mathbb{H}$-valued square integrable random variables;
\item $L_{\mathbb{F}}^{\infty}(0, T ; \mathbb{H})$: the space of $\mathbb{F}$-adapted $\mathbb{H}$-valued essentially bounded processes;
\item $L_{\mathbb{F}}^{\infty, c}( 0,T; \mathbb{H})$: the space of $\mathbb{F}$-adapted $\mathbb{H}$-valued essentially bounded continuous processes;
\item $L_{\mathbb{F}}^2\left(0,T ; \mathbb{H}\right)$: the space of $\mathbb{F}$-adapted processes $X:[0, T] \times \Omega \rightarrow \mathbb{H}$ satisfying
$$ \mathbb{E}\left[\int_0^T\|X(s)\|^2 \ds \right] < \infty;$$
\item $L_{\mathbb{F}}^{2,c}(0,T; \mathbb{H})$: the space of $\mathbb{F}$-adapted continuous processes $X:[0, T] \times$ $\Omega \rightarrow \mathbb{H}$ satisfying
\begin{equation*}
\mathbb{E}\left[\sup _{0 \leq s \leq T}\|X(s)\|^2\right]<\infty.
\end{equation*}
\end{itemize}

We will also use the concept of bounded mean oscillation (BMO) martingale. For a process $\phi \in L_{\mathbb{F}}^2(0, T ; \mathbb{H})$, the stochastic integral $\int_0^{\cdot} \phi(s) \dw(s)$ is called a BMO martingale on $[0, T]$ if its $L_{\mathbb{F} }^{2,\mathrm{bmo}}$ norm on $[0, T]$ is finite, i.e.,
$$
\left\|\int_0^{\cdot} \phi(s)\dw(s)\right\|_{L_{\mathbb{F} }^{2,\mathrm{bmo}}} :=\sup _{\tau \leq T}\left(\operatorname{ess} \sup \mathbb{E}\left[\int_\tau^T\|\phi(s)\|^2 \ds \;\bigg|\; \mathcal{F}_\tau\right]\right)^{\frac{1}{2}}<\infty.
$$
Here and in the rest of paper the supremum $\sup _{\tau \leq T}$ is taken over all $\mathbb{F}$-stopping times $\tau \leq T$. With above norm, we define the space
\begin{equation*}
	L_{\mathbb{F}}^{2, \mathrm{bmo}}(0, T ; \mathbb{H}):=\left\{\phi \in L_{\mathbb{F}}^2(0, T ; \mathbb{H}) : \int_0^{\cdot} \phi(s) \dw(s) ~\text {is a BMO martingale on }[0, T]\right\}.
\end{equation*}

\subsection{Assumptions on the coefficients}
We next introduce the assumptions on the coefficients of {\bf Problem \ref{MFLQ}}.
\begin{assumption} \label{A1}
It holds that
\begin{equation*}
A, C \in L_{\mathbb{F}}^{\infty}( \R^{n \times n}), \quad B, D \in L_{\mathbb{F}}^{\infty}( \R^{n \times m}),\quad a, b\in L_{\mathbb{F}}^{2}(0,T; \R^n).
\end{equation*}
\end{assumption}
\begin{assumption} \label{A2}
(Standard case) It holds that
\begin{equation*}
Q\in L_{\mathbb{F}}^{\infty}\left(\mathbb{S}_{+}^n\right),\ R\in L_{\mathbb{F}}^{\infty}\left(\mathbb{S}_{+}^m\right),\ \bar{G} \in L^{\infty}_{\mathcal{F}_T}\left(\mathbb{S}^n\right),\ G \in L_{\mathcal{F}_T}^{\infty}\left(\mathbb{S}_{+}^n\right),\ q,p\in L_{\mathbb{F}}^{2}(0,T; \R^n), \ \xi, \eta \in L_{\mathcal{F}_{T}}^2( \R^n).
\end{equation*}
Moreover, there exists a constant $\delta>0$ such that $ R\geq \delta I_m$.
\end{assumption}

\begin{assumption} \label{A3}
(Singular case) It holds that
\begin{equation*}
Q\in L_{\mathbb{F}}^{\infty}\left(\mathbb{S}_{+}^n\right),\ R\in L_{\mathbb{F}}^{\infty}\left(\mathbb{S}_{+}^m\right),\ \bar{G} \in L^{\infty}_{\mathcal{F}_T}\left(\mathbb{S}^n\right),\ G \in L_{\mathcal{F}_T}^{\infty}\left(\mathbb{S}_{+}^n\right),\ q,p\in L_{\mathbb{F}}^{2}(0,T; \R^n), \ \xi, \eta \in L_{\mathcal{F}_{T}}^2( \R^n).
\end{equation*}
Moreover, there exists a constant $\delta>0$ such that $G\geq \delta I_n$ and $D^{\top} D\geq \delta I_n$.
\end{assumption}

\begin{assumption}\label{A4}
There exists a constant $\delta>0$ such that $G \geq \delta I_n$.
\end{assumption}

By relabeling $(u-p, a+Bp, b+Dp)\to (u,a,b)$, in what follows, we take $p\equiv0$ in the cost functional \eqref{functional} for convenience. Also, to avoid heavy notations, we sometimes omit the explicit dependence of processes on the time variable and write $\mathbb{E}$ instead of $\mathbb{E}^{0,x,u}$ if there is no risk of confusion.

Notice that
\begin{equation} \label{eqn1}
	\begin{aligned}
&~ \mathbb{E}^{0,x,u} \langle\bar{G} (\mathbb{E}^{0,x,u}[X(T)]-\eta), \mathbb{E}^{0,x,u}[X(T)]-\eta \rangle\\
=&~\langle\mathbb{E}^{0,x,u}[\bar{G}] \mathbb{E}^{0,x,u}[X(T)], \mathbb{E}^{0,x,u}[X(T)] \rangle -2\langle\mathbb{E}^{0,x,u}[\bar{G}\eta], \mathbb{E}^{0,x,u}[X(T)] \rangle
+\mathbb{E}^{0,x,u} \langle\bar{G} \eta, \eta \rangle,
\end{aligned}
\end{equation}
Hence, without loss of generality, we consider the following functional instead of \eqref{eqn1},
\begin{equation*}
	\langle\bar{G} \mathbb{E}^{0,x,u}[X(T)], \mathbb{E}^{0,x,u}[X(T)] \rangle - 2 \langle \zeta,\mathbb{E}^{0,x,u}[X(T)] \rangle,
\end{equation*}
where $\zeta = \mathbb{E}^{0,x,u}[\bar{G}\eta]$ is a constant vector, and $\bar{G}$ is a non-zero constant matrix from now on.

\begin{remark}
If one further assumes $G$ is uniformly positive or both $R$ and $Q$ are non-singular, then one can take $\eta=0$ since the term involving $\eta$ may be easily absorbed by
\begin{align*}
\mathbb{E}^{0,x,u} \left\{ \langle G (X(T)-\xi),X(T)-\xi\rangle \right\},
\end{align*}
or, thanks to It\^o's lemma,
\begin{align*}
&~\langle\mathbb{E}^{0,x,u}[\bar{G}\eta], \mathbb{E}^{0,x,u}[X(T)] \rangle\\
=&~\langle\mathbb{E}^{0,x,u}[\bar{G}\eta], x\rangle+
\mathbb{E}^{0,x,u} \int_0^T \langle\mathbb{E}^{0,x,u}[\bar{G}\eta], A(s) X(s) + B(s) u(s) + a(s) \rangle \ds.
\end{align*}
\end{remark}

\subsection{Bridge problems}

It is the mean-field term $ \langle \bar{G} \mathbb{E}^{0,x,u}[X(T)], \mathbb{E}^{0,x,u}[X(T)]\rangle$ in the cost functional \eqref{functional} that makes {\bf Problem \ref{MFLQ}} out of the scope of the classical SLQ theory.
Our idea is to introduce several bridge problems so that {\bf Problem \ref{MFLQ}} can be tackled by the classical SLQ theory.

The key is to remove the mean-field term from the problem. To this end,
we introduce the constrained control set for a fixed vector $d \in \R^n$ as below
$$
\mathcal{U}^d[0, T] := \left\{ u \in \mathcal{U}[0, T]:\ \mathbb{E}[X(T)] = d \right\}.
$$
Since for any $u \in \mathcal{U}^d[0, T]$, the only mean-field term $\langle \bar{G} \mathbb{E}^{0,x,u}[X(T)], \mathbb{E}^{0,x,u}[X(T)]\rangle$ in {\bf Problem \ref{MFLQ}} is a constant, so we can focus the classical part of the problem.
The set $\mathcal{U}^d[0, T]$ may be empty for some $d$, so we need some condition to ensure it is non-empty. Such a condition is stated as follows.

\begin{lemma} \label{feasible-condition}
Under {\bf Assumption \ref{A1}}, the set $\mathcal{U}^d[0, T]$ is non-empty for all $d \in \R^n$ if and only if
\begin{equation} \label{posi-definite}
\mathbb{E}\int_{0}^{T} \left\|Z(t) D(t) + Y(t)B(t) \right\| \dt > 0,
\end{equation}
where $(Y,Z) \in L_{\mathbb{F}}^{\infty}\left(0, T ; \R^{n \times n}\right) \times L_{\mathbb{F}}^{2}(0, T ; \R^{n \times n})$ is the unique solution to linear matrix-valued BSDE
\begin{equation*}
\left\{\begin{aligned}
\dd Y(t) &= - \left[ Y(t) A(t) +Z(t) C(t) \right] \dt +Z(t) \dw(t), \quad t \in[0, T], \\
Y(T) &= I_n.
\end{aligned}\right.
\end{equation*}
Moreover, if the condition \eqref{posi-definite} does not hold, then the set $\mathcal{U}^d[0, T]$ is empty except for one $d \in \R^n$.
\end{lemma}
\begin{proof}
The proof is similar to Theorem 4.1 in Zhang and Zhang \cite{zhang2023stochastic}, in which the feasibility condition for the homogeneous multi-dimensional state equation is proved. The idea of proof is also similar to Theorem 5.3 in Hu, Shi and Xu \cite{MR4729330}, in which the non-homogeneous one-dimensional case is studied. So we omit the arguments here.
\end{proof}

When the condition \eqref{posi-definite} does not hold, the above result shows that $\mathbb{E}^{0,x,u}[X(T)]$ is a constant for all controls $u \in \mathcal{U}[0, T]$, so the mean-field term
$ \langle \bar{G} \mathbb{E}^{0,x,u}[X(T)], \mathbb{E}^{0,x,u}[X(T)]\rangle$ in the cost functional \eqref{functional} does not play a role and {\bf Problem \ref{MFLQ}} becomes a classical SLQ problem. To avoid this degenerate case, we assume the condition \eqref{posi-definite} holds true in the rest of this paper without claim.

The following SLQ problem with control constraint serves as a bridge to study {\bf Problem \ref{MFLQ}}.
\begin{prob} \label{constrained-SLQ}
For any given $x, d \in \R^n$, find $\bar{u} \in \mathcal{U}^d[0, T]$ such that $J\bigl(x ; \bar{u}\bigr)= V(x,d)$ where
$$
V(x,d) := \inf_{u \in \mathcal{U}^d[0, T]} J(x ; u).
$$
\end{prob}
Once the above problem is solved, we can solve the original {\bf Problem \ref{MFLQ}} by selecting the best terminal expectation $d$.

A standard approach to deal with the constrained {\bf Problem \ref{constrained-SLQ}} is to introduce a Lagrange multiplier $\lambda$ and apply the well-known Lagrange duality theorem. To see this, for any given $x, d, \lambda \in \R^n$ and $u \in \mathcal{U}[0, T]$, we define a new cost functional
\begin{align} \label{J1def}
J_1(x,d,\lambda ; u) :=&~\mathbb{E}\int_0^T \Bigl[ \langle Q(s) (X(s)-q(s)), X(s)-q(s) \rangle + \langle R(s)u(s), u(s) \rangle \Bigr] \ds \nn\\
& + \mathbb{E} \Bigl\{ \langle G (X(T)-\xi), X(T)-\xi \rangle + \langle \bar{G} d, d \rangle -2 \langle \zeta,d \rangle\Bigr\} + 2 \langle \lambda, \mathbb{E}[X(T)]-d \rangle\nn\\
=&~\mathbb{E}\bigg\{\int_0^T \Bigl[ \langle Q(s) (X(s)-q(s)), X(s)-q(s) \rangle + \langle R(s)u(s), u(s) \rangle \Bigr] \ds \nn\\
& + \langle G (X(T)-\xi), X(T)-\xi \rangle + 2 \langle \lambda, X(T)\rangle \bigg\} + \langle\bar{G} d, d \rangle - 2 \langle \zeta+ \lambda, d\rangle.
\end{align}
This cost functional does not involve any mean-field terms, so the following unconstrained stochastic optimal control problem can be tackled by classical theory.
\begin{prob} \label{slq}
For any given $x, d, \lambda \in \R^n$, find $\bar{u} \in \mathcal{U}[0, T]$ such that $J_1\bigl(x,d,\lambda ; \bar{u}\bigr)=V_1(x,d,\lambda)$, where
$$
V_1(x,d,\lambda) := \inf_{u \in \mathcal{U}[0, T]} J_1(x,d,\lambda ; u).
$$
\end{prob}

By \eqref{J1def}, we see the optimal control to { \bf Problem \ref{slq}} does not depend on $d$, and it holds that
\begin{align*}
V_1(x,d,\lambda)=V_1(x,0,\lambda)+ \langle\bar{G} d, d \rangle - 2 \langle \zeta+ \lambda, d\rangle,
\end{align*}
for any given $x, d, \lambda \in \R^n$.
In addition, the Lagrange duality theorem gives
\begin{align*}
V(x,d)=\sup_{\lambda \in \R^n}V_1(x,d,\lambda).
\end{align*}

To move forward with our problem, we need {\bf Assumption \ref{A4}}, which allows us to define a new cost functional
\begin{equation*}
\begin{aligned}
J_2(x,\lambda ; u) :=&~ \mathbb{E} \bigg\{\int_0^T \left[ \langle Q(s) (X(s)-q(s)), X(s)-q(s) \rangle +\langle R(s)u(s), u(s) \rangle \right] \ds
\\ &\qquad\qquad + \langle G (X(T)-\xi+G^{-1}\lambda ),X(T)-\xi +G^{-1}\lambda \rangle\bigg\}.
\end{aligned}
\end{equation*}
Note this cost functional differs from $J_1(x,d,\lambda; u)$ by only a constant, so the following optimal control problem is equivalent to { \bf Problem \ref{slq}}.
\begin{prob} \label{SLQ}
For any given $x,\lambda \in \R^n$, find $\bar{u} \in \mathcal{U}[0, T]$ such that
$J_2\left(x,\lambda ; \bar{u}\right)=V_2(x,\lambda)$, where
\begin{equation*}
V_2(x,\lambda) : = \inf _{u \in \mathcal{U}[0, T]} J_2(x,\lambda ; u).
\end{equation*}
\end{prob}
Under standard assumptions, {\bf Problems \ref{slq}} and {\bf \ref{SLQ} } admit the same unique optimal control. Moreover, it is straightforward to verify that
\begin{equation} \label{equality}
V_1(x,d,\lambda) = V_2(x,\lambda) + 2 \langle \mathbb{E}[ \xi], \lambda \rangle - \langle \mathbb{E} [G^{-1}]\lambda,\lambda \rangle + \langle\bar{G} d, d \rangle - 2 \langle \zeta+ \lambda, d\rangle.
\end{equation}

{ \bf Problems \ref{constrained-SLQ}}, {\bf \ref{slq}} and {\bf \ref{SLQ}} are all regarded as bridge problems to solve the original {\bf Problem \ref{MFLQ}}.
Since they are SLQ problems, we may solve them using existing theory. In particular, we study their SRE first.

\section{Solution to \textbf{Problem \ref{MFLQ}}}\label{sec3}
We will solve \textbf{Problem \ref{MFLQ}} in this section by solving the bridge problems first.
In Subsection \ref{positiveSRE}, we show the SRE \eqref{bsde1} admits a uniformly positive solution.
Then use it to solve the bridge problems and \textbf{Problem \ref{MFLQ}} in Subsection \ref{SREsolution}.

\subsection{Uniformly Positive Solution for SRE}\label{positiveSRE}

We present a new result on the existence of a uniformly positive solution for the SRE \eqref{bsde1} in this section.

In the singular case, Kohlmann and Tang \cite{kohlmann2002global} and \cite{kohlmann2003minimization} proved that the SRE \eqref{bsde1} admits unique positive solutions in one-dimensional and multi-dimensional cases, respectively,
so we only need to resolve the standard case.
Clearly, to let the solution be uniformly positive, it is necessary to assume that the terminal condition $G$ is uniformly positive, i.e., {\bf Assumption \ref{A4}} holds.
\begin{theorem} \label{SRE}
Under {\bf Assumptions \ref{A1}}, {\bf \ref{A2}} and {\bf \ref{A4}}, SRE \eqref{bsde1} admits a unique solution $(P, \Lambda) \in L_{\mathbb{F}}^{\infty,c}(0, T; \mathbb{S}_{+}^n) \times L_{\mathbb{F}}^{2,\mathrm{bmo}}(0, T; \mathbb{S}^n)$. Moreover, $P$ is uniformly positive.
\end{theorem}
\begin{proof}
The solvability of multi-dimensional SRE \eqref{bsde1} was given by Tang \cite{tang2003general}, so it only remains to show that $P$ is uniformly positive. The key step is to derive an appropriate priori estimate for $\mathbb{E}^{\mathcal{F}_t}\|X(\cdot)\|^2$ on $[t,T]$.

Now we fix an arbitrary time $t\in[0,T)$ and set $X(t)=x\in\R^n$.
Applying It\^{o}'s formula, we obtain
\begin{equation*}
\begin{aligned}
\mathbb{E}^{\mathcal{F}_t}\langle X(T),X(T) \rangle =&~ \mathbb{E}^{\mathcal{F}_t}\langle X(s), X(s)\rangle + 2 \mathbb{E}^{\mathcal{F}_t} \int_s^T \langle A X+B u, X \rangle\dr +\mathbb{E}^{\mathcal{F}_t} \int_s^T \|C X+D u \|^2 \dr \\
\geq &~ \mathbb{E}^{\mathcal{F}_t}\|X(s)\|^2 - \alpha \mathbb{E}^{\mathcal{F}_t} \int_s^T\|u\|^2 \dr-\beta \mathbb{E}^{\mathcal{F}_t} \int_s^T\|X\|^2 \dr,~~s\in[t,T],
\end{aligned}
\end{equation*}
where the constant $\alpha>0$ depends only on the coefficients $B, C, D$ and the constant $\beta>0$ depends on the coefficients $A, B, C, D$. Set
\begin{equation*}
\rho_s:=\mathbb{E}^{\mathcal{F}_t}\|X(s)\|^2\ \ {\rm and}\ \ \kappa_s := \mathbb{E}^{\mathcal{F}_t}\int_s^T\|u\|^2 \dr, \qquad s\in[t,T].
\end{equation*}
The above estimate, together with the fact $\kappa_s\leq \kappa_t$, leads to
\begin{equation*} 
\rho_s\leq \rho_T+ \alpha \kappa_t + \beta \int_s^T \rho_r \dr,\qquad s\in[t,T].
\end{equation*}
By Gronwall's inequality it yields
\begin{equation}\label{ineq1}
\rho_s \leq (\rho_T + \alpha \kappa_t)e^{\beta(T-s)} \leq ( 1 \vee \alpha)(\rho_T + \kappa_t)e^{\beta T},\qquad s\in[t,T].
\end{equation}
By the dynamic programming principle for general SLQ optimal control problems (see Theorem 4.6 in Tang \cite{tang2015dynamic}), we have
\begin{equation*}
\langle P(t)x,x \rangle= \essinf _{u \in \mathcal{U}[t, T]} \mathbb{E}^{t,x,u}\left[\left\langle G X(T), X(T)\right\rangle+\int_t^T\left(\left\langle Q(s) X(s), X(s)\right\rangle+\left\langle R(s) u(s), u(s)\right\rangle \right) \ds\right].
\end{equation*}
Then it follows from {\bf Assumptions \ref{A1}}, {\bf\ref{A2}}, {\bf \ref{A4}} and \eqref{ineq1} that
\begin{equation*}
\begin{aligned}
\langle P(t)x,x \rangle \geq &~ \delta \mathbb{E}^{\mathcal{F}_t}\| X(T)\|^2 + \delta \mathbb{E}^{\mathcal{F}_t} \int_t^T\|u\|^2 \ds =\delta (\rho_T + \kappa_t)
\geq \frac{\delta e^{-\beta T}}{1 \vee \alpha} \rho_t = \frac{\delta e^{-\beta T}}{1 \vee \alpha} \|x\|^2.
\end{aligned}
\end{equation*}
Since $x\in\R^n$ is arbitrary chosen, we obtain $\essinf P(t) \geq \frac{\delta e^{-\beta T}}{1 \vee \alpha} I_{n}$ for any fixed $t \in [0,T]$.
Bearing in mind the fact that $P$ admits continuous trajectories, we conclude that $P$ is uniformly positive.
\end{proof}

\begin{remark}
In particular, when $n=1$, under {\bf Assumptions \ref{A1}}, {\bf\ref{A2}} and {\bf \ref{A4}}, it is well-known that the SRE \eqref{bsde1} admits a uniformly positive solution; see, e.g. Hu, Shi and Xu \cite[Lemma 3.1]{MR4729330}.

\end{remark}

\subsection{Solution to bridge problems and {\bf Problem \ref{MFLQ}}}\label{SREsolution}
We start from considering {\bf Problem \ref{slq}}, namely an unconstrained non-homogeneous SLQ optimal control problem. To characterize its value function and the associated optimal feedback control, for $x = (x_1, x_2,\dots,x_n)^{\top}\in \R^n$ and $\lambda = (\lambda_1,\lambda_2,\dots,\lambda_n)^{\top} \in \R^n$, we introduce a linear and non-homogeneous BSDE involving unbounded coefficients
\begin{equation} \label{bsde2}
\left\{\begin{aligned}
& \dd H = -g_2(H,\Gamma) \dt + \Gamma \dw, \quad t \in[0, T], \\
& H(T)= \lambda - G\xi,
\end{aligned}\right.
\end{equation}
where
\begin{gather*}
g_2(H,\Gamma) :=~ \hat{A}^{\top}H + \hat{C}^{\top}\Gamma + ( \hat{C}^\top P + \Lambda )b + Pa - Qq,\\
\hat{A} :=~ A + B\Theta(P,\Lambda), \qquad \hat{C} := C+D\Theta(P,\Lambda),
\end{gather*}
and
$(P, \Lambda) \in L_{\mathbb{F}}^{\infty,c}(0, T; \mathbb{S}_{+}^n) \times L_{\mathbb{F}}^{2,\mathrm{bmo}}(0, T; \mathbb{S}^n)$ is the unique solution to SRE \eqref{bsde1}, here $\Theta$ is defined in \eqref{thetafun}.

The solvability of BSDE \eqref{bsde2} follows immediately from Theorem 3.11 in Kohlmann and Tang \cite{kohlmann2003minimization}.
\begin{proposition} \label{unboundedbsde}
Under {\bf Assumptions \ref{A1}} and {\bf \ref{A2}} or {\bf Assumptions \ref{A1}} and {\bf\ref{A3}}, the multi-dimensional linear BSDE \eqref{bsde2} admits a unique solution $(H, \Gamma) \in L_{\mathbb{F}}^{2,c}\left(0, T ; \R^n\right) \times L_{\mathbb{F}}^{2}\left(0, T ; \R^n\right)$.
\end{proposition}
\begin{remark}
Actually, there are existing well-posedness results for linear BSDEs in general forms involving unbounded coefficients in $L_{\mathbb{F} }^{2,\mathrm{bmo}}(0, T ; \R^{n \times n})$. See Hu, Shi and Xu \cite[Lemma 3.6]{MR4729330} and Wang, Xu and Zhang \cite[Lemma A.1]{WXZ24} for one-dimensional equation and Chen and Luo \cite[Theorems 2.8 and 3.1]{MR4893237} for multi-dimensional equation.
\end{remark}

Moreover, referring to the proof of Theorem 3.8 in Kohlmann and Tang \cite{kohlmann2003minimization}, we can obtain the optimal feedback control and value function of {\bf Problem \ref{slq}}, with the help of BSDEs \eqref{bsde1} and \eqref{bsde2}.
\begin{theorem}\label{valuefunction}
Assume {\bf Assumptions \ref{A1}} and {\bf \ref{A2}} hold. Denote respectively by $(P, \Lambda)$ and $(H, \Gamma)$ the solutions to BSDEs \eqref{bsde1} and \eqref{bsde2}. Then
the linear SDE
\begin{equation*}\label{sxz1}
\left\{\begin{aligned}
\dd \bar{X}(s) =&~ \left\{ \hat{A}(s) \bar{X}(s) + B(s)u_0(s) \right\} \ds + \left\{ \hat{C}(s) \bar{X}(s) + D(s)u_0(s) \right\} \dw(s), \quad s\in[0,T], \\
\bar{X}(0) =&~ x\in\R^n,
\end{aligned}\right.
\end{equation*}
where
\begin{equation*}
u_0 = - \left(D^{\top} P D + R\right)^{-1}\left( B^{\top}H+D^{\top}\Gamma + D^\top Pb \right),
\end{equation*}
admits a unique solution $\bar{X} \in L_{\mathbb{F}}^{2,c}(0,T; \R^n)$.
Moreover, the optimal feedback control of {\bf Problem \ref{slq}} is
\begin{equation*}
\bar{u} =\Theta(P,\Lambda) \bar{X} + u_0,
\end{equation*}
and the value function of {\bf Problem \ref{slq}} is
\begin{equation*}
V_1(x,d,\lambda) = J_1\left(x,\lambda ; \bar{u}\right) = \langle P(0) x, x \rangle + 2 \langle H(0), x \rangle + \langle\bar{G} d, d \rangle - 2 \langle \zeta+ \lambda, d\rangle+ c,
\end{equation*}
where
\begin{equation*}
c: = \mathbb{E}\left\{ \langle G\xi, \xi \rangle + \int_{0}^{T}\left[ \langle Qq,q\rangle + \langle Pb,b \rangle +2\langle H,a\rangle + 2\langle \Gamma, b\rangle - \langle (D^{\top} P D + R) u_0, u_0\rangle \right]\ds \right\},
\end{equation*}
is a constant independent of $(x,d,\lambda)$.
\end{theorem}

Next, we aim to optimize the value function $V_1(x,d,\lambda)$ with respect to the multiplier $\lambda$ which appears in the terminal condition of BSDE \eqref{bsde2}. Since $H$ depends on $\lambda$, we need to express the relationship explicitly.
For this, we decompose the linear BSDE \eqref{bsde2} into the sum of a linear non-homogeneous BSDE:
\begin{equation*}
\left\{\begin{aligned}
& \dd H^0 = - g_2(H^0,\Gamma^0) \dt + \Gamma^0 \dw, \quad t \in[0, T], \\
& H^0(T)= -G\xi,
\end{aligned}\right.
\end{equation*}
and $n$ linear homogeneous BSDEs:
\begin{equation*}
\left\{\begin{aligned}
& \dd H^i = - \left(\hat{A}^{\top}H^i + \hat{C}^{\top}\Gamma^i \right) \dt + \Gamma^i \dw, \quad t \in[0, T], \quad i=1,2,\cdots,n,\\
& H^i(T)= e_i,
\end{aligned}\right.
\end{equation*}
where $e_i$ is a basis vector in $\mathbb{R}^n$ whose $i$-th component is $1$ and all other components are $0$.
Note all $(H^i,\Gamma^i)$, $0 \leq i \leq n$, are independent of $\lambda$ and $d$.

By the uniqueness of the solution to BSDE \eqref{bsde2}, it is not hard to verify $H = H^0 + \sum_{i=1}^{n} \lambda_i H^i $ and $\Gamma = \Gamma + \sum_{i=1}^{n} \lambda_i \Gamma^i$.
As a consequence, the value function $V_1(x,d,\lambda)$ can be reformulated as
\begin{equation*}
V_1(x,d,\lambda) = -\langle \Psi \lambda, \lambda \rangle + 2 \langle \psi - d, \lambda \rangle + \Delta + \langle \bar{G} d, d \rangle - \langle \zeta,d \rangle,
\end{equation*}
where $\Psi={[\Psi (i,j)]}_{i,j=1,2,\cdots,n} \in \R^{n \times n}$, $\psi=(\psi(1),\psi(2),\cdots,\psi(n))^{\top}\in \R^n$ and $\Delta\in \R$ are defined as
\begin{align*}
\Psi (i,j) =&~ \mathbb{E} \int_{0}^{T} \Big\langle \left(D^{\top} P D + R\right)^{-1} \left( B^{\top}H^j+D^{\top} \Gamma^j \right), B^{\top}H^i+D^{\top}\Gamma^i \Big\rangle \ds,\\
\psi(i) =&~ \langle H^i(0),x
\rangle + \mathbb{E} \int_{0}^{T} \left(\langle H^i, a \rangle + \langle \Gamma^i, b\rangle - \langle (D^{\top} P D + R)^{-1}(B^{\top}H^i+D^\top\Gamma^i), D^\top Pb \rangle \right) \ds,\\
\Delta =&~ \mathbb{E}\bigg\{ \langle G\xi, \xi \rangle + \int_{0}^{T}\Big[ \langle Qq,q\rangle + \langle Pb,b \rangle + 2\langle H^0, a \rangle + 2\langle \Gamma^0, b\rangle \\
&\qquad- \langle (D^{\top} P D + R)^{-1} D^\top Pb, D^\top Pb \rangle \Big]\ds \bigg\} + \langle P(0) x, x \rangle + 2 \langle H^0,x\rangle.
\end{align*}
Obviously, $\Psi$, $\psi$ and $\Delta$ are constants and independent of $(x,d,\lambda)$.

We now show the solvability of {\bf Problem \ref{MFLQ}} under the first type of sufficient condition.

\begin{theorem} \label{main1}
Under {\bf Assumptions \ref{A1}} and {\bf \ref{A2}}, the matrix $\Psi $ is positive definite. Moreover, {\bf Problem \ref{MFLQ}} admits a unique optimal pair if $\bar{G} + \Psi^{-1} > 0$.
\end{theorem}
\begin{proof}
We first prove that $\Psi$ is positive semidefinite. Indeed, for any vector $z = (z_1, z_2,\cdots,z_n)^{\top}\in \R^n$, we have
\begin{equation*}
\begin{aligned}
z^{\top}\Psi z& = \sum_{i,j=1}^{n}\Psi(i,j)z_i z_j = \mathbb{E}\int_{0}^{T} \left( B^{\top}\hat{H}+D^{\top}\hat{\Gamma} \right)^{\top} \left(D^{\top} P D + R\right)^{-1} \left( B^{\top}\hat{H}+D^{\top} \hat{\Gamma} \right) \dt \geq 0,
\end{aligned}
\end{equation*}
where $(\hat{H}, \hat{\Gamma})$ solves BSDE
\begin{equation} \label{positive-condition}
\left\{\begin{aligned}
& \dd \hat{H} = - \left(\hat{A}^{\top}\hat{H} + \hat{C}^{\top}\hat{\Gamma} \right) \dt + \hat{\Gamma} \dw, \quad t \in[0, T], \\
& \hat{H}(T)= z.
\end{aligned}\right.
\end{equation}
Hence $\Psi$ is positive semidefinite.

According to the classical Lagrange duality theorem (see, e.g. Luenberger \cite{luenberger1997optimization}), for any $d \in \R^n$, the functional $J_1$ is strongly convex w.r.t $(x,u)$ under {\bf Assumptions \ref{A1}} and {\bf \ref{A2}}, so the following Lagrange duality holds on
\begin{equation*}
V(x,d) = \max _{\lambda \in \R^n} V_1(x,d,\lambda).
\end{equation*}
The gradient and Hessian of $V_1(x,d,\lambda)$ with respect to $\lambda$ are
\begin{equation*}
\nabla_{\lambda} V_1(x,d,\lambda) = -2\Psi\lambda + 2\psi - 2 d\ \ {\rm and}\ \ \text{Hess}(V_1(x,d,\lambda)) = -2\Psi.
\end{equation*}
Since the function $V_1$ has extremum points from the duality theorem, so at any extremum point $\lambda^*$, $\nabla_{\lambda} V_1(x,d,\lambda^*)=0$, that is,
\begin{equation*}
\Psi \lambda^* = \psi - d.
\end{equation*}
Since $d$ is arbitrarily chosen, it shows that the image of $\Psi$ is $ \R^n$, and thus $\Psi$ is nonsingular. Since $\Psi$ is positive semidefinite, we conclude $\Psi$ is positive definite.

Furthermore, the optimal Lagrange multiplier is uniquely given by
\begin{equation*}
\lambda^* = \Psi^{-1}\left(\psi -d\right).
\end{equation*}
As a consequence, the value function of the constrained SLQ {\bf Problem \ref{constrained-SLQ}} is
\begin{equation*}
V(x,d) = \max _{\lambda \in \R^n} V_1(x,d,\lambda) = V_1(x,d,\lambda^*) = \langle \Psi^{-1}\left(\psi -d\right), \psi -d \rangle + \Delta + \langle \bar{G} d, d \rangle - \langle \zeta,d \rangle .
\end{equation*}
Moreover, if $\bar{G} + \Psi^{-1} > 0$, then $V(x,d)$ is a positive definite functional of $d$ which admits a unique minimizer $d^*$, given by
\[d^*=\frac{1}{2}(\bar{G} + \Psi^{-1})^{-1}(\Psi^{-1}\psi +\zeta).\]
Consequently, we have the optimal Lagrange multiplier
\begin{equation*}
\lambda^* = \Psi^{-1}\left(\psi -d^*\right),
\end{equation*}
and its associated optimal state feedback control $\bar{u}$.
\end{proof}

The conclusions, as well as their proofs, of Theorems \ref{valuefunction} and \ref{main1} remain valid if \textbf{Assumption \ref{A2}} (the standard case) is replaced by \textbf{Assumption \ref{A3}} (the singular case). The key point in the proofs of Theorems \ref{valuefunction} and \ref{main1} under \textbf{Assumptions \ref{A1}} and \textbf{\ref{A3}} is the well-posedness of the SRE \eqref{bsde1}, which fortunately exists in Tang \cite{tang2003general} and \cite{tang2015dynamic}.

However, the explicit form of $\Psi^{-1}$ is not clear in most cases. Therefore, we provide an alternative method to study \textbf{Problem \ref{MFLQ}}, which offers a more intuitive view of our optimal control problem. Although the alternative method requires a structural condition --- namely, that the solution $P$ to the SRE \eqref{bsde1} must be positive rather than merely nonnegative --- the positivity of $P$ has already been established by Theorem \ref{SRE}.

We start from {\bf Problem \ref{SLQ}} and introduce a non-homogeneous linear BSDE with possibly unbounded coefficients
\begin{equation} \label{bsde3}
\left\{\begin{aligned}
& \dd K = - g_3 (K,\Xi) \dt + \Xi \dw, \quad t \in[0, T], \\
& K(T)= G^{-1} \lambda - \xi,
\end{aligned}\right.
\end{equation}
where
\begin{equation*}
\begin{aligned}
g_3 (K,\Xi) = \left( -A -P^{-1}Q - P^{-1} \mathscr{S} C \right)K +P^{-1}\mathscr{S}\Xi + \left(a + P^{-1}\mathscr{S} b -P^{-1}Qq \right)\ \ {\rm and}\ \ \mathscr{S} = \hat{C}^\top P+ \Lambda.
\end{aligned}
\end{equation*}

Indeed, the well-posedness of BSDE \eqref{bsde3} comes from the well-posedness of BSDE \eqref{bsde2}.
\begin{theorem} \label{transformation}
Under {\bf Assumptions \ref{A1}}, {\bf \ref{A2}} and {\bf \ref{A4}}
or {\bf Assumptions \ref{A1}} and {\bf\ref{A3}},
BSDE \eqref{bsde3} admits a unique solution $(K,\Xi)\in L_{\mathbb{F}}^{2,c}\left(0, T; \R^n\right) \times L_{\mathbb{F}}^{2}\left(0, T; \R^n\right)$.
\end{theorem}

\begin{proof}
Assume {\bf Assumptions \ref{A1}}, {\bf \ref{A2}} and {\bf \ref{A4}} hold.
By the It\^{o} formula for matrix inverse, the dynamics equation for $\hat{P}:=P^{-1}$ is
$$
\left\{
\begin{aligned}
& \dd \hat{P} = \left( \hat{P} g_1(P,\Lambda) \hat{P} + \hat{P} \Lambda \hat{P} \Lambda \hat{P} \right) \dd t - \hat{P} \Lambda \hat{P} \dd W, \quad t \in [0,T], \\
& \hat{P}(T) = G^{-1}.
\end{aligned}
\right.
$$
By Proposition \ref{unboundedbsde}, BSDE \eqref{bsde2} admits a unique adapted solution $(H, \Gamma) \in L_{\mathbb{F}}^{2,c}(0,T; \R^n) \times L_{\mathbb{F}}^{2}(0,T; \R^n)$.
Define
$$
{K} = \hat {P} H\ \ {\rm and}\ \ {\Xi} = \hat{P} \big( \Gamma - \Lambda \hat{P} H \big).
$$
Applying the It\^{o} formula again, we know that $({K}, {\Xi}) \in L_{\mathbb{F}}^{2,c}(0,T; \R^n) \times L_{\mathbb{F}}^{2}(0,T; \R^n)$ solves BSDE \eqref{bsde3}, and the existence of BSDE \eqref{bsde3} follows.

If there is another solution $(\tilde{K}, \tilde{\Xi}) \in L_{\mathbb{F}}^{2,c}(0,T; \R^n) \times L_{\mathbb{F}}^{2}(0,T; \R^n)$, set
$$
H = P \tilde{K}, \quad \Gamma = \Lambda \tilde{K} + P \tilde{\Xi},
$$
which gives another solution to BSDE \eqref{bsde2}. Then the uniqueness of solution to BSDE \eqref{bsde3} follows from the uniqueness of solution to BSDE \eqref{bsde2}.

The proof is similar when {\bf Assumptions \ref{A1}} and {\bf\ref{A3}} hold.
\end{proof}


Similar to the Theorem \ref{valuefunction}, we can obtain the optimal feedback control and value function of {\bf Problem \ref{SLQ}} based on BSDEs \eqref{bsde1} and \eqref{bsde3} in the following theorem, referring to the proof of Lemma 4.2 in Hu, Shi and Xu \cite{MR4729330}.
\begin{theorem} \label{extension}
Under {\bf Assumptions \ref{A1}}, {\bf \ref{A2}}, and {\bf \ref{A4}} or {\bf Assumptions \ref{A1}} and {\bf\ref{A3}}, the linear SDE
\begin{equation*}\label{sxz2}
\left\{\begin{aligned}
\dd \tilde{X}(s) =&~\big\{ \hat{A}(s) \tilde{X}(s) + B(s) u_1(s) \big\} \ds
+ \big\{ \hat{C}(s) \tilde{X}(s) + D(s) u_1(s) \big\} \, \dd W(s), & s \in [0,T], \\
\tilde{X}(0) =&~x,
\end{aligned}\right.
\end{equation*}
where
$$
u_1 := -\left(D^{\top} P D + R\right)^{-1} \big( B^\top P K + D^\top \Lambda K + D^\top P \Xi \big),
$$
admits the unique solution $\tilde{X}$. Moreover,
the optimal feedback control of {\bf Problem \ref{SLQ}} is
$$
\tilde{u} = \Theta(P,\Lambda) \tilde{X} + u_1,
$$
and the value function of {\bf Problem \ref{SLQ}} is
$$
V_2(x,\lambda) = J_2\big(x,\lambda ; \bar{u}\big) = \langle P(0) (x + K(0)), x + K(0) \rangle + \mathbb{E}\Big[ \int_{0}^{T} l(t) \, \dd t \Big],
$$
where
$$
l = \langle \mathscr{R} (-C K + \Xi + b), -C K + \Xi + b \rangle + \langle Q (K+q), K+q \rangle\ \ {\rm and}\ \ \mathscr{R} = P - P D \left(D^{\top} P D + R\right)^{-1} D^{\top} P.
$$
\end{theorem}
\begin{remark}
Theorems \ref{transformation} and \ref{extension} can be regarded as a generalization of the result in Section 5.3 of Kohlmann and Tang \cite{kohlmann2002global} to the multi-dimensional case.
\end{remark}

Theorem \ref{extension} explicitly characterizes both the optimal feedback control and the value function in terms of the solutions $(P, \Lambda)$ and $(K, \Xi)$ to the associated BSDEs \eqref{bsde1} and \eqref{bsde3}. Indeed, by the relationships
$$
H = P K \ \ {\rm and}\ \ \Gamma = \Lambda K + P \Xi,
$$
we can directly verify that the dynamics of $\tilde{X}$ in Theorem \ref{valuefunction} coincide with those of $\bar{X}$ in Theorem \ref{extension}, and consequently the controls satisfy $u_1 = u_0$. It also indicates that the optimal feedback controls $\tilde{u} = \Theta(P, \Lambda) \tilde{X} + u_1 = \bar{u}$ and the equivalence between {\bf Problems \ref{slq}} and {\bf \ref{SLQ}}.

Moreover, we can separate the dependence on the Lagrange multiplier $\lambda$ from the non-homogeneous part by decomposing BSDE \eqref{bsde3} into
\begin{equation}
\left\{\begin{aligned}
& \dd K^0 = - g_3\left( K^0,\Xi^0 \right) \dt + \Xi^0 \dw, \quad t \in[0, T], \\
& K^0(T)= -\xi,
\end{aligned}\right.
\end{equation}
and
\begin{equation}
\left\{\begin{aligned}
& \dd K^i = - \left( \left( -A -P^{-1}Q - P^{-1} \mathscr{S} C \right)K +P^{-1}\mathscr{S}\Xi \right) \dt + \Xi^i \dw, \quad t \in[0, T], \\
& K^i(T)= G^{-1} e_i.
\end{aligned}\right.
\end{equation}
Obviously, $K = K^0 + \sum_{i=1}^{n} \lambda_i K^i $ and $\Xi = \Xi^0 + \sum_{i=1}^{n} \lambda_i \Xi^i$. Consequently, the value function $V_1(x,d,\lambda)$ can be written in the quadratic form
\begin{equation*}
V_1(x,d,\lambda) = \langle (\Phi-\mathbb{E} [G^{-1}]) \lambda,\lambda \rangle + 2 \langle \phi-d, \lambda \rangle + \delta +\langle \bar{G} d, d \rangle - \langle \zeta,d \rangle,
\end{equation*}
where $\Phi={[\Phi (i,j)]}_{i,j=1,2,\cdots,n} \in \R^{n \times n}$, $\phi=(\phi(1),\phi(2),\cdots,\phi(n))^{\top}\in \R^n$ and $\delta\in \R$ are defined as
\begin{align}
\Phi(i,j) =&~ \langle P(0)K^i(0),K^j(0) \rangle + \mathbb{E}\left[ \int_{0}^{T} \langle \mathscr{R} ( -C K^i + \Xi^i), - C K^j + \Xi^j \rangle + \langle QK^i,K^j \rangle \dt \right], \label{defPhi}\\
\phi(i) = &~\langle P(0)(x+K^0(0)),K^i(0) \rangle \nn\\
&~+ \mathbb{E}\left[ \int_{0}^{T} \langle \mathscr{R} ( -C K^0 + \Xi^0+b), - C K^i + \Xi^i \rangle + \langle Q(K^0+q),K^i \rangle \dt \right] + \mathbb{E}[ \xi](i), \nn\\
\delta =&~ \langle P(0)(x+K^0(0)),x+K^0(0) \rangle \nn\\
&~+ \mathbb{E}\left[ \int_{0}^{T} \langle \mathscr{R} ( -C K^0 + \Xi^0+b), - C K^0 + \Xi^0+b \rangle + \langle Q(K^0+q),K^0+q \rangle \dt \right].\nn
\end{align}
Here $\Phi$, $\phi$ and $\delta$ are determined by the solutions $(K^0, \Xi^0)$ and $(K^i, \Xi^i)$, $i=1,2,\cdots,n$. This decomposition gives a chance to straightforward optimize $V_1(x,d,\lambda)$ with respect to $\lambda$ since the quadratic structure guarantees a unique maximum.

\begin{remark}
The identity
$$
z^\top \left( P - P D (D^{\top} P D + R)^{-1} D^{\top} P \right) z = \inf_{v \in \R^m} \left\{ (z - D v)^\top P (z - D v) + v^\top R v \right\} \ge 0,
$$
indicates that $\mathscr{R} = P - P D (D^\top P D + R)^{-1} D^\top P$ is positive semidefinite. Hence $l\ge0$ as a sum of nonnegative terms.
\end{remark}

We now show the solvability of {\bf Problem \ref{MFLQ}} under the second type of sufficient condition.

\begin{theorem} \label{main2}
Under {\bf Assumptions \ref{A1}}, {\bf\ref{A2}}, and {\bf\ref{A4}} or {\bf Assumptions \ref{A1}} and {\bf\ref{A3}}, we have $\Phi - \mathbb{E}[G^{-1}]<0$. Moreover, {\bf Problem \ref{MFLQ}} is solvable if $\bar{G} - (\Phi - \mathbb{E}[G^{-1}])^{-1} > 0$.
\end{theorem}
\begin{proof}
The proof is similar to Theorem \ref{main1}.
Assume {\bf Assumptions \ref{A1}}, {\bf \ref{A2}} and {\bf \ref{A4}} hold. The gradient of $V_1(x,d,\lambda)$ is
\begin{equation*}
\nabla_{\lambda} V_1(x,d,\lambda) = 2(\Phi-\mathbb{E} [G^{-1}]) \lambda + 2(\phi-d),
\end{equation*}
which yields the optimal Lagrange multiplier
\begin{equation*}
\lambda^* = - ({\Phi-\mathbb{E} [G^{-1}]}) ^{-1}(\phi-d).
\end{equation*}
As a result, the value function of the constrained SLQ {\bf Problem \ref{constrained-SLQ}} is
\begin{equation*}
V(x,d) = \max _{\lambda \in \R^n} V_1(x,d,\lambda) = V_1(x,d,\lambda^*) = - \langle ({\Phi-\mathbb{E} [G^{-1}]}) ^{-1}(\phi-d),\phi-d \rangle + \delta + \langle \bar{G} d, d \rangle - \langle \zeta,d \rangle.
\end{equation*}
Moreover, with the condition $\bar{G} - (\Phi - \mathbb{E}[G^{-1}])^{-1} > 0$, an optimal $d^*$ exists, from which the corresponding $\lambda^*$ and optimal control $\bar{u}$ can be determined.

The proof is similar when {\bf Assumptions \ref{A1}} and {\bf\ref{A3}} hold.
\end{proof}



Until now we have established two types of sufficient conditions given in Theorem \ref{main1} and Theorem \ref{main2} to ensure the solvability of {\bf Problem \ref{MFLQ}}.

We next provide three specific sufficient conditions that can guarantee the second type sufficient condition $\bar{G} - (\Phi - \mathbb{E}[G^{-1}])^{-1} > 0$ in Theorem \ref{main2} holds so that {\bf Problem \ref{MFLQ}} is solvable.

\begin{proposition}\label{condition2positive}
Under {\bf Assumptions \ref{A1}}, {\bf\ref{A2}}, and {\bf\ref{A4}} or {\bf Assumptions \ref{A1}} and {\bf\ref{A3}},
it holds that $\bar{G} - (\Phi - \mathbb{E}[G^{-1}])^{-1} > 0$ and thus {\bf Problem \ref{MFLQ}} is solvable under each of the following conditions.
\begin{enumerate}
\item $\bar{G} + \mathbb{E}[G^{-1}]^{-1}> 0$;
\item $\Phi>0$ and $\bar{G} + \mathbb{E}[G^{-1}]^{-1}\geq 0$;
\item $\Phi>0$, $G$ is deterministic and
\begin{equation*}
\bar{G}> - G^{\frac{1}{2}} \left(I_{n} - G^{\frac{1}{2}} \Phi G^{\frac{1}{2}} \right)^{-1} G^{\frac{1}{2}}.
\end{equation*}

\end{enumerate}
\end{proposition}

\begin{proof}
Write $\hat{G}:= \mathbb{E}[G^{-1}]>0$, then by Theorem \ref{main2}, we have
\begin{equation*}
	\hat{G}^{\frac{1}{2}} \left[ \hat{G}^{-\frac{1}{2}} \Phi \hat{G}^{-\frac{1}{2}} - I_{n} \right]\hat{G}^{\frac{1}{2}} = \Phi - \mathbb{E}[G^{-1}] < 0,
\end{equation*}
which implies
\begin{equation*}
	I_{n} \geq I_{n}- \hat{G}^{-\frac{1}{2}} \Phi \hat{G}^{-\frac{1}{2}}>0,
\end{equation*}
and
\begin{equation*}
	\left(I_{n}- \hat{G}^{-\frac{1}{2}} \Phi \hat{G}^{-\frac{1}{2}} \right)^{-1} \geq I_{n}.
\end{equation*}
Hence
\begin{align}\label{est1}
-\left( \Phi- \mathbb{E}[G^{-1}]\right)^{-1} - \mathbb{E}[G^{-1}]^{-1} = \hat{G}^{-\frac{1}{2}} \left[ \left(I_{n}- \hat{G}^{-\frac{1}{2}} \Phi \hat{G}^{-\frac{1}{2}} \right)^{-1}-I_{n} \right]\hat{G}^{-\frac{1}{2}} \geq 0.
\end{align}
This implies $\bar{G} - (\Phi - \mathbb{E}[G^{-1}])^{-1} > 0$ when $\bar{G} + \mathbb{E}[G^{-1}]^{-1}> 0$.

When $\Phi>0$, all the inequalities in the preceding argument are strict. Therefore, we have $\bar{G} - (\Phi - \mathbb{E}[G^{-1}])^{-1} > 0$ holds when $\bar{G} + \mathbb{E}[G^{-1}]^{-1}\geq 0$. Under the last condition, we have $\mathbb{E}[G^{-1}]^{-1}=G$
and $\hat{G}^{-\frac{1}{2}}=G^{\frac{1}{2}}$,
so
\begin{align*}
\bar{G}- (\Phi - \mathbb{E}[G^{-1}])^{-1} &>
- G^{\frac{1}{2}} \left(I_{n} - G^{\frac{1}{2}} \Phi G^{\frac{1}{2}} \right)^{-1} G^{\frac{1}{2}}-(\Phi - \mathbb{E}[G^{-1}])^{-1}\\
&= - G^{\frac{1}{2}} \left[ \left(I_{n} - G^{\frac{1}{2}} \Phi G^{\frac{1}{2}} \right)^{-1} - I_{n} \right] G^{\frac{1}{2}}-G- (\Phi - \mathbb{E}[G^{-1}])^{-1}=0,
\end{align*}
where the last equation is due to \eqref{est1}.
\end{proof}

\begin{remark}
In the literature, it is often assumed that the coefficients of the problem are all deterministic (which implies $\Phi > 0$ by Proposition \ref{Phi>0} below) and that $\bar{G} + G \geq 0$ (see, e.g., Sun \cite{su} and Yong \cite{yong2013linear,yo2}).
The third condition in Proposition \ref{condition2positive}, which allows $\bar{G} + G$ to be negative definite, is evidently weaker than the assumptions in those existing works.
\end{remark}

The last two conditions in Proposition \ref{condition2positive} require $\Phi$ to be positive definite, which is not easy to verify since $\Phi$ is not a directly given parameter. We next investigate under what conditions $\Phi>0$ holds.

 Define the matrix $\tilde{K} = \{K^1(0), K^2(0), \ldots, K^n(0)\}$. If $\tilde{K}$ is invertible, then $\Phi$ is positive definite due to its definition \eqref{defPhi} and the fact that $P(0) > 0$. Notice that
\begin{equation*}
	\hat{H}(0) = \sum_{i=1}^{n} H^i(0) z_i = \sum_{i=1}^{n} P(0) K^i(0) z_i = P(0) \tilde{K} z,
\end{equation*}
so $\tilde{K}$ is invertible if and only if $\hat{H}(0) \neq 0$ for any nonzero $z \in \mathbb{R}^n$, where $(\hat{H}, \Gamma)$ is the solution to the BSDE \eqref{positive-condition}. Hence, our problem reduces to finding suitable conditions to ensure $\hat{H}(0) \neq 0$.
To this end, we introduce the following result, whose proof can be found in Theorems 3.4 and 3.11 of Kohlmann and Tang \cite{kohlmann2003minimization}.

\begin{proposition} \label{relationship}
Under {\bf Assumptions \ref{A1}} and {\bf\ref{A2}} or {\bf Assumptions \ref{A1}} and {\bf\ref{A3}}, $\widehat{u} \in L_{\mathbb{F}}^2\!\left(0,T; \R^m\right)$ is the optimal control of {\bf Problem \ref{slq}} if and only if there exists a pair of processes $(\hat{Y},\hat{Z}) \in L_{\mathbb{F}}^{2,c}(0,T; \R^n)\times L_{\mathbb{F}}^2(0,T; \R^n)$ solving the adjoint equation
\begin{equation} \label{FBSDE}
\left\{\begin{aligned} \dd \hat{Y}(s) & =-\left( A(s)^{\top} \hat{Y}(s)+ C(s)^{\top} \hat{Z}(s)+Q(s) ( \widehat{X}(s)-q)\right) \ds + Z(s) \dw(s), \quad s \in[0, T], \\
\hat{Y}(T) & = G ( \bar{X}(T) - \xi) +\lambda,
\end{aligned}\right.
\end{equation}
where $\widehat{X}\in L_{\mathbb{F}}^{2,c}(0,T; \R^n)$ is the unique solution of the state equation
\begin{equation*}
\left\{\begin{aligned}
\dd \hat{X}(s) &= \{A(s) \hat{X}(s) + B(s) \hat{u}(s) + a\} \ds + \{C(s) \bar{X}(s) + D(s) \hat{u}(s)+b \} \dw(s), \quad s \in[0, T], \\
\hat{X}(0) &= x,
\end{aligned}\right.
\end{equation*}
and the stationary condition
\begin{equation*}
B^{\top} \hat{Y} + D^{\top} \hat{Z} + R \widehat{u} = 0, 
\end{equation*}
holds. Furthermore, the relations involving solutions to BSDE \eqref{bsde1}, BSDE \eqref{bsde2} and BSDE \eqref{FBSDE} hold:
\begin{equation*}
H= - P\hat{X} + \hat{Y}\ \ {\rm and} \ \ \Gamma= -P \left( C\hat{X} + D\hat{u} + b \right) - \Lambda \hat{X} + \hat{Z}.
\end{equation*}
\end{proposition}

For BSDE \eqref{bsde2}, if we replace $\lambda$ by $\lambda+z$, we obtain another linear BSDE
\begin{equation} \label{bsde4}
\left\{
\begin{aligned}
\dd H_1(s) &= -g_2(H_1(s),\Gamma_1(s))\dt + \Gamma_1(s)\dw(s), \quad s\in[0,T], \\[0.5ex]
H_1(T) &= \lambda+z - G\xi,
\end{aligned}\right.
\end{equation}
and BSDE \eqref{positive-condition} can be regarded as the difference between \eqref{bsde2} and \eqref{bsde4}, i.e.,
\begin{equation*}
\hat{H}= H_1 - H \ \ {\rm and} \ \ \hat{\Gamma} = \Gamma_1 - \Gamma.
\end{equation*}
Meanwhile, consider the adjoint equation
\begin{equation*}\label{FBSDE1}
\left\{
\begin{aligned}
\dd \hat{Y}_1(s) &= -\big(A(s)^{\top}\hat{Y}_1(s) + C(s)^{\top}\hat{Z}_1(s) + Q(s)(\widehat{X}(s)-q)\big)\dt + \hat{Z}_1(s)\dw(s), \quad s \in[0, T], \\[0.5ex]
\hat{Y}_1(T) &= G(\widehat{X}(T)-\xi) + \lambda+z,
\end{aligned}\right.
\end{equation*}
and define
\[
\tilde{Y}:=\hat{Y}_1-\hat{Y}\ \ {\rm and} \ \ \tilde{Z}:=\hat{Z}_1-\hat{Z}.
\]
Then $(\tilde{Y},\tilde{Z})$ solves the linear BSDE
\begin{equation}\label{bsde5}
\left\{
\begin{aligned}
\dd \tilde{Y}(s) &= -\big(A(s)^{\top}\tilde{Y}(s) + C(s)^{\top}\tilde{Z}(s)\big)\dt + \tilde{Z}(s)\dw(s),\quad s \in[0, T], \\[0.5ex]
\tilde{Y}(T) &= z \neq 0.
\end{aligned}\right.
\end{equation}
Moreover, we have
\begin{equation*}
\hat{H}(0) = H_1(0) - H(0) = \big(-P(0)\widehat{X}(0)+\hat{Y}_1(0)\big) - \big(-P(0)\widehat{X}(0)+\hat{Y}(0)\big) = \tilde{Y}(0).
\end{equation*}
Hence, $\hat{H}(0) \neq 0$ if and only if $\tilde{Y}(0) \neq 0$.
This allows us to derive the following result.

\begin{proposition} \label{Phi>0}
Under \textbf{Assumptions \ref{A1}} and \textbf{\ref{A2}}, or \textbf{Assumptions \ref{A1}} and \textbf{\ref{A3}}, it holds that $\Phi > 0$ if either the state equation is one-dimensional or $A$ is a deterministic process.
\end{proposition}
\begin{proof}
It suffices to show that $\hat{H}(0) = \tilde{Y}(0) \neq 0$ under the given conditions.

If the state equation is one-dimensional, then $\tilde{Y}(0) \neq 0$ follows immediately from the comparison theorem for scalar-valued BSDEs.

If $A$ is deterministic, then by the uniqueness of solutions to linear BSDEs, we obtain $\tilde{Z} = 0$ and consequently $\tilde{Y}(0) \neq 0$.
\end{proof}

\section{A Numerical Study}\label{numerical}
In this section, we present a numerical example for the multi-asset MV portfolio selection model introduced in Section \ref{sec:introduction}. The aim of this experiment is twofold: first, to illustrate the financial implications of the preference parameters $\upsilon$ and $\Sigma$, and second, to examine the behavior of the value function in a representative market setting.

We consider a financial market consisting of three risky assets: equity, bonds, and commodities. The corresponding price processes are assumed to follow Black-Scholes dynamics, as given in \eqref{assets} for $i=1,2,3$. The model parameters are chosen to reflect typical market conditions commonly adopted in the asset allocation literature. Specifically, the annualized expected returns and volatilities are set to
\begin{equation*}
	(\mu_1,\mu_2,\mu_3)= (0.08,0.03,0.05) \qquad (\sigma_1,\sigma_2,\sigma_3) = (0.20,0.05,0.30).
\end{equation*}

The investor's preference functional \eqref{preference} is characterized by the return weight vector $\upsilon$ and the risk aversion matrix $\Sigma$. These are set to
\begin{equation*}
	\upsilon = (1.5,\; 1.0,\; 0.5)^\top,
\end{equation*}
which specifies the relative importance assigned to the expected terminal wealth of each asset. This choice reflects heterogeneous beliefs or strategic preferences across different assets: the investor places the highest weight on equity returns, a moderate weight on bond returns, and a relatively lower weight on commodity returns.

The risk aversion matrix, which is symmetric and positive definite, is set to
\begin{equation*}
	\Sigma =
	\begin{pmatrix}
		3.5 & 0.6 & -0.5 \\
		0.6 & 2.6 & -0.3 \\
		-0.5 & -0.3 & 1.5
	\end{pmatrix}.
\end{equation*}
We denote by $\Sigma_{ij}$ the element in the $i$-th row and $j$-th column of $\Sigma$. The matrix $\Sigma$ quantifies the investor's sensitivity to risk. Its diagonal elements measure the penalty imposed on the variance of each asset, while the off-diagonal elements capture perceived interaction risks across different assets, thereby encoding diversification and hedging effects. Positive off-diagonal elements penalize comovement and discourage concentrated exposure, whereas negative values promote diversification benefits. In summary, $(\upsilon,\Sigma)$ provides a flexible and financially interpretable extension of the classical scalar mean--variance preference to a multi-asset setting.

The initial positions in the equity, bond and commodity are set to $15$, $10$, and $5$, respectively, i.e.,
\begin{equation*}
	x = (x_1,x_2,x_3) = (15,\,10,\,5)^\top.
\end{equation*}
In the above settings, the objective functional
\begin{equation*}
	\mathscr{J}(x,\pi)
	= \mathbb{E}\,\langle 2\upsilon, X(T) \rangle
	- \mathbb{E}\,\big\langle
	\Sigma \big( X(T)-\mathbb{E}[X(T)] \big),
	X(T)-\mathbb{E}[X(T)]
	\big\rangle,
\end{equation*}
admits a clear financial interpretation. The investor's objective is to maximize $\mathscr{J}(x,\pi)$, and hence the value function is
\begin{equation*}
	\mathcal{J}(x)
	:= \sup_{\pi\in\mathcal{U}[0,T]} \mathscr{J}(x,\pi).
\end{equation*}

Certainly, this problem can also be reformulated to minimize the functional $-\mathscr{J}(x,\pi)$, which fits into the theoretical framework developed for {\bf Problem \ref{MFLQ}}.

In this example, the coefficients of state equation \eqref{staeq} satisfy
\[
	A = a = C = b = 0,
\]
\begin{equation*}
	B =
	\begin{pmatrix}
		0.08 & 0 & 0 \\
		0 & 0.03 & 0 \\
		0 & 0 & 0.05
	\end{pmatrix}
	\ \ {\rm and}\ \
	D =
	\begin{pmatrix}
		0.20 & 0 & 0 \\
		0 & 0.05 & 0 \\
		0 & 0 & 0.25
	\end{pmatrix}.
\end{equation*}
The coefficients of the cost functional \eqref{functional} reduce to
\begin{equation*}
	Q = q = R = p = \eta = 0,
	\ \
	G = \Sigma,
	\ \
	\bar{G} = -\Sigma
	\ \ {\rm and}\ \
	\xi = \Sigma^{-1}\upsilon.
\end{equation*}
It is straightforward to verify that {\bf Assumptions~\ref{A1}} and~{\bf \ref{A3}} are satisfied and that $G+\bar{G}=0$. Therefore, all the well-posedness results established in the previous sections are valid in this example.

Obviously, BSDE \eqref{bsde3} admits a constant solution given by
\begin{equation*}
	(K,\Xi) \equiv (G^{-1}\lambda - \xi,\,0).
\end{equation*}
Consequently, the function $V_2(x,\lambda)$ takes the form
\begin{equation*}
	V_2(x,\lambda)
	= \big\langle
	P(0)\big(x + \Sigma^{-1}\lambda - \xi\big),
	x + \Sigma^{-1}\lambda - \xi
	\big\rangle.
\end{equation*}
According to the relation \eqref{equality}, the function $V_1(x,d,\lambda)$ is given by
\begin{equation*}
	V_1(x,d,\lambda)
	= \langle \Phi \lambda, \lambda\rangle
	+ 2 \langle \phi - d, \lambda \rangle
	- \langle \Sigma d,d \rangle
	+ \langle P(0)(x - \xi), x - \xi \rangle,
\end{equation*}
where
\[
	\Phi = \Sigma^{-1} P(0) \Sigma^{-1} - \Sigma^{-1} < 0\ \ {\rm and}\ \
	\phi = \Sigma^{-1} P(0)(x-\xi) + \xi.
\]
The maximizer with respect to $\lambda$ is
\[
	\lambda^*(d) = \Phi^{-1}(\phi - d).
\]
Substituting $\lambda^*(d)$ into $V_1$ yields
\begin{equation*}
	V(x,d)
	= \max_{\lambda \in \R^3} V_1(x,d,\lambda)
	= \langle (-\Sigma - \Phi^{-1})d,d \rangle
	+ 2 \langle \Phi^{-1}\phi, d \rangle
	- \langle \Phi^{-1}\phi, \phi \rangle
	+ \langle P(0)(x - \xi), x - \xi \rangle,
\end{equation*}
where $-\Sigma - \Phi^{-1} > 0$. The minimizer of $V(x,d)$ is therefore
\[
	d^* = (-\Sigma - \Phi^{-1})^{-1} \Phi^{-1}\phi,
\]
and the value function of {\bf Problem \ref{MFLQ}} is given by
\begin{equation*}
	\min_{d \in \R^3} V(x,d)
	= -\big\langle
	(-\Sigma - \Phi^{-1})^{-1} \Phi^{-1}\phi,
	\Phi^{-1}\phi
	\big\rangle
	- \langle \Phi^{-1}\phi, \phi \rangle
	+ \langle P(0)(x - \xi), x - \xi \rangle.
\end{equation*}
As a result, we obtain
\begin{equation*}
	\mathcal{J}(x)
	= \big\langle
	(-\Sigma - \Phi^{-1})^{-1} \Phi^{-1}\phi,
	\Phi^{-1}\phi
	\big\rangle
	+ \langle \Phi^{-1}\phi, \phi \rangle
	- \langle P(0)(x - \xi), x - \xi \rangle.
\end{equation*}



From the explicit expression of $\mathcal{J}(x)$, it is clear that $\mathcal{J}(x)$ is a quadratic function of the initial state $x$. However, the dependence of $\mathcal{J}(x)$ on the risk aversion matrix $\Sigma$ is not yet clear. We therefore proceed to investigate the impact of $\Sigma$ on the value function $\mathcal{J}(x)$.

To this end, we choose the off-diagonal elements $\Sigma_{21}$, $\Sigma_{31}$ and $\Sigma_{32}$ from the set $\{-0.5,\,0,\,0.5\}$ and compute the corresponding values of $\mathcal{J}(x)$. It turns out that $\mathcal{J}(x)$ attains its maximum when
\[
	\Sigma_{21} = \Sigma_{31} = \Sigma_{32} = -0.5.
\]
Accordingly, we take
\begin{equation*}
	\Sigma =
	\begin{pmatrix}
		3.5 & -0.5 & -0.5 \\
		-0.5 & 2.6 & -0.5 \\
		-0.5 & -0.5 & 1.5
	\end{pmatrix},
\end{equation*}
as the benchmark risk aversion matrix and its corresponding portfolio serves as a baseline one. Starting from this benchmark, we then vary each element $\Sigma_{ij}$ individually while keeping all other parameters, including the preference vector $\upsilon$, unchanged, and plot the resulting values of $\mathcal{J}(x)$. Finally, we obtain Figure \ref{image1} and \ref{image2} corresponding to variations in the diagonal and off-diagonal elements of $\Sigma$, respectively.
\begin{figure}[htbp]
	\centering
	\hspace{0em}
	\begin{minipage}[t]{0.49\textwidth}
	\centering
	\includegraphics[width=20em]{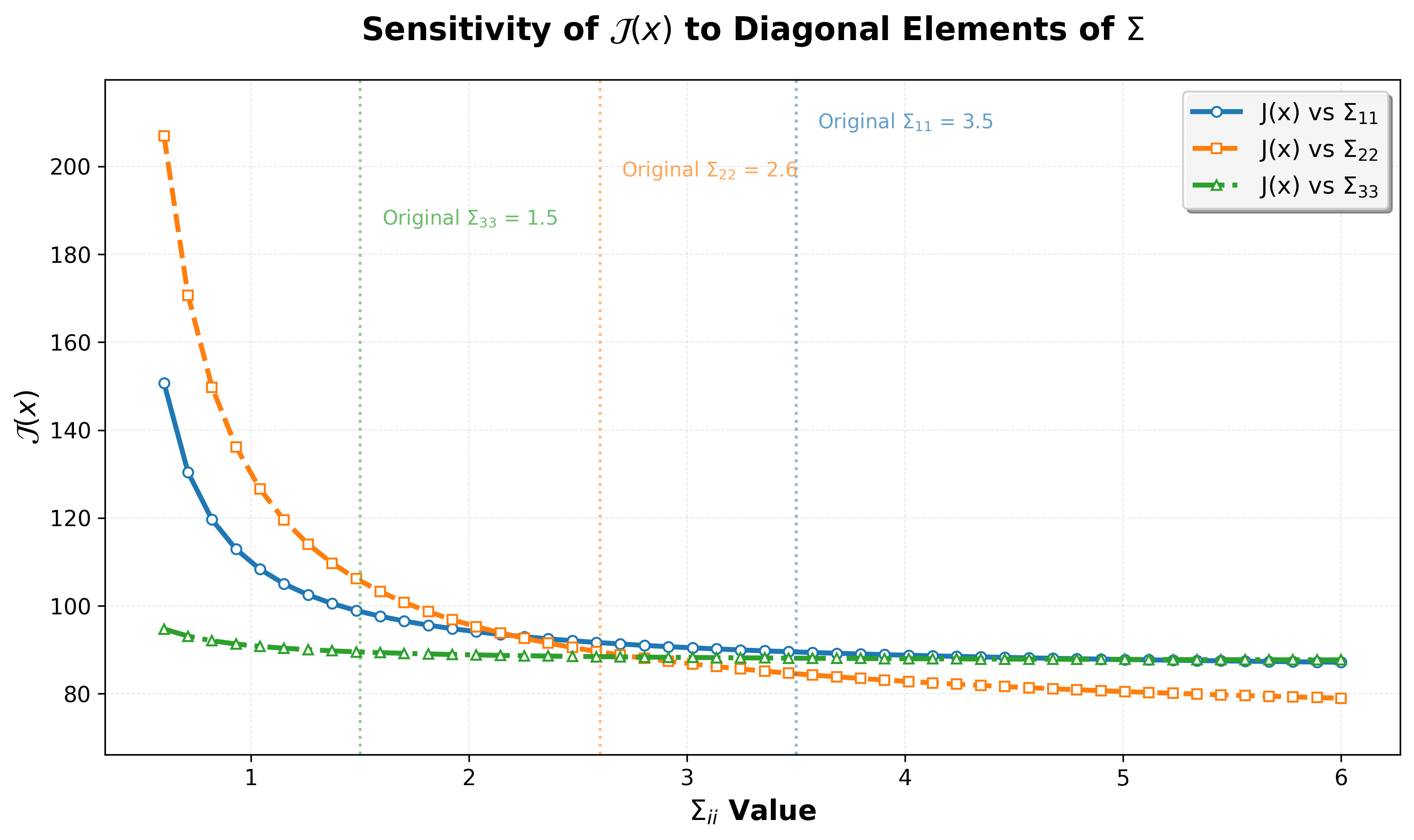}
	\caption{Diagonal}
	\label{image1}
	\end{minipage}
	\begin{minipage}[t]{0.49\textwidth}
	\centering
	\includegraphics[width=20em]{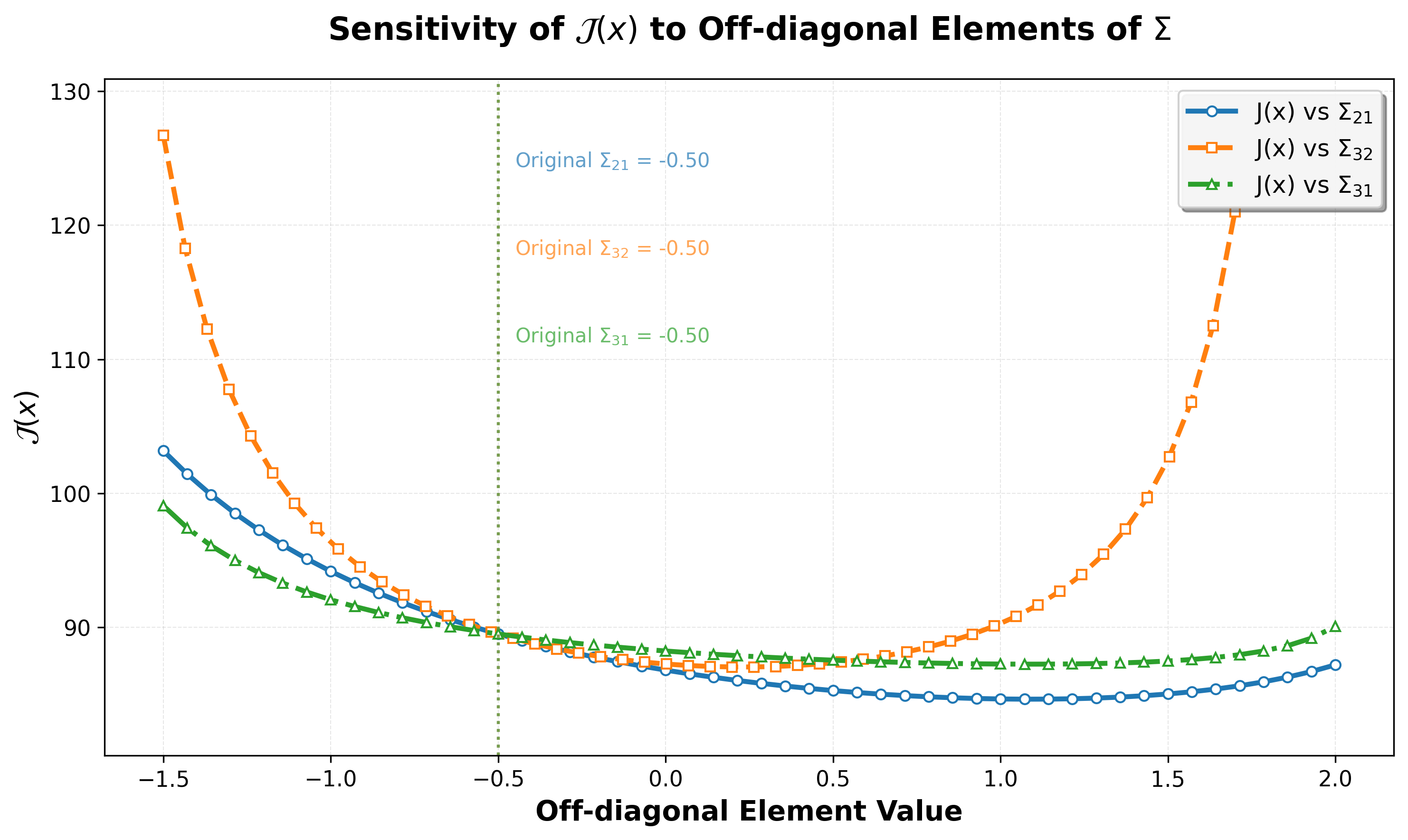}
	\caption{Off-diagonal}
	\label{image2}
	\end{minipage}
\end{figure}
\begin{remark}
	When we vary the elements of $\Sigma$, it is necessary to ensure that the resulting matrix remains positive definite. For instance, if an off-diagonal element $\Sigma_{ij}$ $(i \neq j)$ is changed, the symmetric element $\Sigma_{ji}$ should be adjusted in the meantime.
\end{remark}
Figure \ref{image1} illustrates the sensitivity of the objective value with respect to the diagonal elements $\Sigma_{11}$, $\Sigma_{22}$, and $\Sigma_{33}$. All three curves exhibit a monotonic decreasing pattern, reflecting the fact that increasing the variance penalty of any individual account inevitably leads to a deterioration of the MV performance.

Notably, the curve associated with the bond account shows a significantly steeper decline. From a financial perspective, this phenomenon can be interpreted as imposing a stronger risk penalty on an account that is less risky. Since this account plays a stabilizing role in the baseline portfolio, an increase in its variance effectively compresses the feasible investment space available to the investor. As a result, the portfolio loses an important low-risk allocation channel, which cannot be fully compensated by reallocating wealth to riskier assets. This leads to the most pronounced deterioration in the objective value among the three cases.

The relatively flat response of the objective value $J(x)$ with respect to $\Sigma_{33}$ indicates that the commodity account plays a secondary role in the MV trade-off. Interpreted as a commodity portfolio, this account is mainly used for diversification rather than return generation, and its optimal exposure remains limited. Consequently, increasing the variance penalty on this account only marginally affects the overall portfolio efficiency, leading to much weaker sensitivity compared with the equity and bond accounts.

Figure \ref{image1} illustrates the sensitivity of the objective value $\mathcal{J}(x)$ with respect to the diagonal elements $\Sigma_{11}$, $\Sigma_{22}$, and $\Sigma_{33}$. All three curves display a monotonic decreasing pattern, reflecting the fact that increasing the variance penalty on any individual asset inevitably harms the MV performance.

It is worth noting that the curve associated with the bond asset exhibits a significantly steeper decline. From a financial perspective, this behavior can be interpreted as the consequence of imposing a stronger risk penalty on an intrinsically low-risk asset. Since the position in bonds plays a stabilizing role in the baseline portfolio, an increase in its variance penalty effectively compresses the feasible investment space available to the investor. As a result, the portfolio loses an important low-risk allocation channel, which cannot be fully compensated by reallocating wealth to riskier assets, thereby leading to the most significant deterioration in the objective value among the three cases.

In comparison, the relatively flat response of $\mathcal{J}(x)$ with respect to $\Sigma_{33}$ indicates that the position in commodities plays a secondary role in the MV trade-off. This asset is primarily held for diversification rather than profit generation, and the risk associated with it is very limited. Consequently, increasing its variance penalty only marginally affects the overall efficiency of the portfolio, resulting in much weaker sensitivity compared with the equity or bond assets.

Figure \ref{image2} exhibits a clear $U$-shaped pattern for all three curves: the objective value $\mathcal{J}(x)$ is relatively large when the off-diagonal elements $\Sigma_{ij}$ $(i \neq j)$ are either strongly negative or strongly positive, and reaches its minimum at an intermediate level of correlation. Taking $\Sigma_{23}$ as an example, when $\Sigma_{23}$ is close to zero, the weak interaction between bonds and commodities limits effective risk sharing and reduces diversification benefits, leading to a lower achievable objective value. As $|\Sigma_{23}|$ increases, stronger dependence allows the portfolio to exploit joint movements more effectively; the investor can construct strategies that improve the trade-off between expected return and variance, resulting in higher objective values.

Although all three curves exhibit a similar $U$-shaped pattern, their rates of variation differ substantially. The curves corresponding to $\Sigma_{21}$ and $\Sigma_{31}$ change relatively slowly and remain fairly flat over a wide range of parameter values, while the curve corresponding to $\Sigma_{32}$ decreases and increases more steeply. This can be explained by the close correlation between bonds and commodities through inflation and interest rate channels.

Despite the $U$-shaped pattern, extreme values of cross-asset correlations are not advisable. There are two main reasons. First, when the correlation is either extremely weak or extremely strong, the effective investment space is compressed, reducing the attainable MV performance. Second, if the correlation is mispriced, the resulting utility loss can be substantial. Thus, moderate correlation levels provide the most balanced trade-off between diversification and coordinated risk allocation. This observation suggests that, from a MV perspective, cross-asset correlations should be neither neglected nor overstated, as extreme specifications inevitably lead to efficiency losses.



The above numerical results provide practical advice for portfolio selection. The variance penalty on individual assets and the cross-asset correlations strongly influence the MV performance. This highlights the importance of the $( \upsilon, \Sigma )$ choice in achieving an efficient and balanced multi-asset portfolio.

%
%

\section{Concluding Remarks}\label{conclude}

This paper is concerned with a general non-homogeneous stochastic linear-quadratic optimal control problem featuring a terminal mean-field term and random coefficients. To address the difficulties caused by the terminal mean-field term, we employ the Lagrangian duality method by first fixing the terminal expectation. We then establish two types of sufficient conditions on the coefficients to guarantee the solvability of the problem and derive the corresponding optimal state-feedback controls. The third specific condition in Proposition \ref{condition2positive} is shown to be weaker than the usual assumptions found in the existing literature. Finally, we construct a concrete financial example involving three risky assets and analyze the impact of $\upsilon$ and $\Sigma$ on utility, thereby providing guidance for the choice of model parameters and investment decisions.

Several interesting issues remain for future investigation. For instance:
\begin{enumerate}
 \item Extending the present method to backward stochastic linear-quadratic optimal control problems with an initial mean-field term in BSDEs and random coefficients;
 \item Integrating the above framework with reinforcement learning techniques to improve the selection of risky asset parameters in numerical examples such as those in Section \ref{numerical}.
\end{enumerate}

\bibliographystyle{amsplain}
%
\renewcommand{\baselinestretch}{1.2}
\providecommand{\bysame}{\leavevmode\hbox to3em{\hrulefill}\thinspace}
\providecommand{\MR}{\relax\ifhmode\unskip\space\fi MR }
\providecommand{\MRhref}[2]{%
\href{http://www.ams.org/mathscinet-getitem?mr=#1}{#2}
}
\providecommand{\href}[2]{#2}

\end{document}